\documentclass[11pt]{amsart}

\usepackage{amsmath,amsthm, amscd, amssymb, amsfonts}
\usepackage{epsfig}
\input xy
\xyoption{all}

\newcommand{\otk}{{\otimes_{\ku}}}
\newcommand{\Mo}{{\mathcal M}}

\newcommand{\nic}{{\mathfrak B}}
\newcommand{\qs}{{\mathfrak R}}

\newcommand{\Ss}{{\mathcal S}}
\newcommand{\ot}{{\otimes}}

\newcommand{\kc}{{\mathcal K}}
\newcommand{\Ac}{{\mathcal A}}

\newcommand{\ca}{{\mathcal C}}

\newcommand{\Do}{{\mathcal D}}
\newcommand{\Pc}{{\mathcal P}}

\newcommand{\YD}{{\mathcal YD}}

\newcommand{\op}{\rm{op}}

\newcommand{\ku}{{\Bbbk}}

\newcommand{\Na}{{\mathbb N}}

\newcommand{\uno}{{\mathbf 1}}

\newcommand{\id}{\mbox{\rm id\,}}

\newcommand\Rep{\operatorname{Rep}}

\newcommand\co{\operatorname{co}}

\newcommand{\End}{\operatorname{End}}

\newcommand{\gr}{\mbox{\rm gr\,}}

\renewcommand{\_}[1]{\mbox{$_{\left( #1 \right)}$}}

\theoremstyle{plain}

\numberwithin{equation}{section}

\newtheorem{teo}{Theorem}[section]

\newtheorem{lema}[teo]{Lemma}

\newtheorem{prop}[teo]{Proposition}

\newtheorem{claim}{Claim}[section]

\theoremstyle{definition}

\newtheorem{defi}[teo]{Definition}

  \newtheorem{exa}[teo]{Example}

\theoremstyle{remark}

\newtheorem{rmk}[teo]{Remark}

\def\pf{\begin{proof}}

\def\epf{\end{proof}}

\theoremstyle{remark}

\begin{document}

\title[Module categories over quantum linear
spaces]{Representations of tensor categories coming from quantum
linear spaces}
\author[Mombelli]{
Mart\'\i n Mombelli }
\thanks{This work was  supported by
 CONICET, Argentina}
\address{ Facultad de Matem\'atica, Astronom\'\i a y F\'\i sica,
Universidad Nacional de C\'ordoba, CIEM, Medina Allende s/n,
(5000) Ciudad Universitaria, C\'ordoba, Argentina}
\email{mombelli@mate.uncor.edu, martin10090@gmail.com
\newline \indent\emph{URL:}\/ http://www.mate.uncor.edu/\~{}mombelli}
\begin{abstract} Exact indecomposable module categories over the tensor category
of representations of Hopf algebras that are liftings of quantum
linear spaces are classified.
\end{abstract}


\date{\today}
\maketitle

\section{Introduction}

Given a tensor category $\ca$, a very natural object to consider
is the family of its \emph{representations}, or \emph{module
categories}. A module category over a tensor category $\ca$ is an
Abelian category $\Mo$ equipped with an exact functor $\ca\times
\Mo\to \Mo$ subject to natural associativity and unity axioms. In
some sense the notion of module category over a tensor category is
the categorical version of the notion of module over an algebra.
In some works the concept of module category over the tensor
category of representations of a quantum group is treated as an
idea more closely related to the notion of \emph{quantum subgroup}
\cite{Oc}, \cite{KO}.

\medbreak

The language of module categories has proven to be a useful tool
in di\-fferent contexts, for example in the theory of fusion
categories, see \cite{ENO1}, \cite{ENO2}, in the theory of weak
Hopf algebras \cite{O1}, in describing some properties of
semisimple Hopf algebras \cite{N} and in relation with dynamical
twists over Hopf algebras \cite{M1} inspired by ideas of V.
Ostrik.

Despite the fact that the notion of module category seems very
general, it is implicitly present in diverse areas of mathematics
and mathematical physics such as subfactor theory \cite{BEK},
affine Hecke algebras \cite{BO}, extensions of vertex algebras
\cite{KO} and conformal field theory, see for example \cite{BFRS},
\cite{FS}, \cite{CS1}, \cite{CS2}.

\medbreak

In \cite{eo} Etingof and Ostrik propose a class of module
categories, called \emph{exact}, and as an interesting problem the
classification of such module categories over a given finite
tensor category. The first classification results  were obtained
in \cite{KO},  \cite{eo2}, where the authors classify semisimple
module categories over the semisimple part of the category of
representations of $U_q(sl(2))$ for a root of unity $q$, over the
category of corepresentations of $SL_q(2)$ in the case $q$ is not
a root of unity and over the fusion category obtained as a
semisimple subquotient of the same category in the case $q$ is a
root of unity. The main result in those papers is the
classification in terms of ADE type Dynkin diagrams, which can be
interpreted as a quantum analogue of the McKay's correspondence.
The classification for the category of corepresentations of
$SL_q(2)$ in the case $q$ is a root of unity was obtained later in
\cite{O3} where the results were quite similar as in the
semisimple case.

\medbreak

In the case when $\ca=\Rep(H)$ is the category of representations
of a finite-dimensional Hopf algebra $H$ the first results
obtained in the classification of module categories were when the
Hopf algebra $H=\ku G$ is the group algebra of a finite group $G$,
see \cite{O1}, and in the case when $H=\Do(G)$ is the Drinfeld's
double of a finite-group $G$, see \cite{O2}. Moreover, in
\emph{loc. cit}. the author classify semisimple module categories
over any group-theoretical fusion category. In \cite{eo} module
categories were classify in the case where $H=T_q$ is the Taft
Hopf algebra, and also for tensor categories of representations of
finite supergroups.

\smallbreak

In \cite{AM},  \cite{M2} the authors give the first steps towards
the understanding of exact module categories over the
representation categories of an arbitrary finite-dimensional Hopf
algebra. In \cite{M2} the author present a technique to classify
module categories over $\Rep(H)$ when $H$ is a finite-dimensional
pointed Hopf algebra inspired by the classification results
obtained in \cite{eo}. In particular a classification is obtained
when $H=\textbf{r}_q$ is the Radford Hopf algebra and when
$H=u_q(\mathfrak{sl}_2)$ is the Frobenius-Lusztig kernel
associated to $\mathfrak{sl}_2$.

\smallbreak

The main goal of this paper is the application of the technique
presented in \cite{M2} to classify exact indecomposable module
categories over representation categories of finite-dimensional
pointed Hopf algebras constructed from quantum linear spaces.

Namely, let $\Gamma$ be a finite Abelian group and $V$  a quantum
linear space in $ {}^{\ku \Gamma}_{\ku \Gamma} \mathcal{YD}$,
$U=\nic(V)\#\ku \Gamma$ the Hopf algebra obtained by bosonization
of the Nichols algebra $\nic(V)$ and $\ku \Gamma$. Then if $\Mo$
is an exact indecomposable module category over $\Rep(U)$  there
exists
\begin{itemize}
    \item a subgroup $F\subseteq \Gamma$,
    \item a normalized 2-cocycle $\psi\in H^2(F,\ku^{\times})$,
    \item a $\ku \Gamma$-subcomodule $W\subseteq V$ invariant
    under the action of $F$,
\item scalars $\xi=(\xi_{i})$, $\alpha=(\alpha_{ij})$ compatible
with $V$, $\psi$, and $F$,
\end{itemize}
such that $\Mo\simeq {}_{\Ac(W,F, \psi, \xi, \alpha)}\Mo$ is the
category of left modules over the left $U$-comodule algebra
$\Ac(W,F, \psi, \xi, \alpha)$ associated to these data. We also show that module categories ${}_{\Ac(W,F, \psi, \xi, \alpha)}\Mo$, ${}_{\Ac(W',F', \psi', \xi', \alpha')}\Mo$ are equivalent as module categories over $\Rep(U)$ if and only if $(W,F, \psi, \xi, \alpha)=(W',F', \psi', \xi', \alpha').$

\smallbreak

If $H$ is a lifting of $U$, that is a Hopf algebra such that the
associated graded Hopf algebra $\gr H$ is isomorphic to $U$, then
$H$ is a cocycle deformation of $U$, implying that the categories $\Rep(H^*)$ and $\Rep(U^*)$ are tensor equivalent.
Thus exact indecomposable module categories over $\Rep(H)$ are
described by the same data as above.

\medbreak

The organization of the paper is as follows. In Section \ref{qls}
we recall the definitions of quantum linear spaces and the
construction of Andruskiewitsch and Schneider of liftings over
quantum linear spaces. In Section \ref{modc} we recall the
definitions of exact module categories and the description of
module categories over finite-dimensional Hopf algebras.

In subsection \ref{techniq} we explain the technique developed in
\cite{M2} to describe exact indecomposable module categories over
$\Rep(H)$ where $H$ is a finite-dimensional pointed Hopf algebra.
The main result is stated as Theorem \ref{descomposition}.

In section \ref{mod-overqs} we present a family of module
categories constructed explicitly over the representation category
of a Hopf algebra constructed from bosonization of a quantum
linear space and a group algebra. Then in Theorem
\ref{clasification-qspaces} we show that any module category is
equivalent to one of this family. Proposition
\ref{generation-in-degree-1} is a key result to the proof of the
main result of this section. In subsection \ref{moritaeq} we prove that any two of those module categories are nonequivalent.

Finally, in section \ref{twist-corresp} we show an explicit
correspondence of comodule algebras over cocycle equivalent Hopf
algebras. Since any lifting of a quantum linear space is a cocycle deformation to the Hopf algebra constructed from this quantum
linear space, this is Proposition \ref{cocy-def}, the results obtained in  Section \ref{mod-overqs}
allows to describe also exact module categories over those
liftings.

\subsection*{Acknowledgments} The author thanks C\'esar
Galindo for pointing out some errors in a previous version of this paper and for some enjoyable and interesting conversations. He also thanks the referee for his constructive comments.

\subsection{Preliminaries and notation}

We shall denote by $\ku$ an algebraically closed field of
characteristic zero. All vector spaces, algebras and categories
will be considered over $\ku$. For any algebra $A$,  ${}_A\Mo$
will denote the category of finite-dimensional left $A$-modules.

\smallbreak

If $\Gamma$ is a finite Abelian group and $\psi\in
Z^2(\Gamma,\ku^{\times})$ is a 2-cocycle, we shall denote  by
$\psi_g$ the map defined by
$$\psi_g(h)=\psi(h,g)\psi(g,h)^{-1},$$
for any $g, h\in\Gamma$. Hereafter we shall assume that any
2-cocycle $\psi$ is normalized and satisfies $\psi(g^{-1},g)=1$
for all $g\in\Gamma$.

\medbreak

If $A$ is an $H$-comodule algebra via $\lambda:A\to H\otk A$, we
shall say that a (right) ideal $J$ is $H$-costable if
$\lambda(J)\subseteq H\otk J$. We shall say that $A$ is (right)
$H$-simple, if there is no nontrivial (right) ideal $H$-costable
in $A$.

\medbreak

If $H$ is a finite-dimensional Hopf algebra then $H_0\subseteq H_1
\subseteq \dots \subseteq H_m=H$ will denote the coradical
filtration. When $H_0\subseteq H$ is a Hopf subalgebra then the
associated graded algebra $\gr H$ is a coradically graded Hopf
algebra. If $(A, \lambda)$ is a left $H$-comodule algebra, the
coradical filtration on $H$ induces a filtration on $A$, given by
$A_n=\lambda^{-1}(H_n\otk A)$. This filtration is called the
\emph{Loewy series }on $A$.

\medbreak

Let $U=\oplus^m_{i=0} U(i)$ be a coradically graded Hopf algebra.
We shall say that a left $U$-comodule algebra $G$, with comodule
structure given by $\lambda:G\to U\otk G$, graded as an algebra
$G=\oplus^m_{i=0} G(i)$ is a \emph{graded comodule algebra} if for
each $0\leq n\leq m$
$$ \lambda(G(n))\subseteq \bigoplus^m_{i=0} U(i)\otk G(n-i).$$
A graded comodule algebra $G=\oplus^m_{i=0} G(i)$ is
\emph{Loewy-graded} if the Loewy series is given by
$G_n=\oplus^n_{i=0} G(i)$ for any $0\leq n\leq m$.

\medbreak

If $A$ is a left $H$-comodule algebra the  graded algebra $\gr A$
obtained from the Loewy series is a Loewy-graded left $\gr
H$-comodule algebra. For more details see \cite{M2}.

\medbreak

We shall need the following result. Let $U=\oplus^m_{i=0} U(i)$ be
a coradically graded Hopf algebra.
\begin{lema}\label{loewy=grad} Let $(A, \lambda)$ be a
left $U$-comodule algebra with an algebra filtration $A^0\subseteq
A^1 \subseteq \dots \subseteq A^m=A$ such that $A_0$ is semisimple
and
\begin{equation}\label{filtr-comod} \lambda(A^n)\subseteq \sum^n_{i=0}\, U(i) \otk
A^{n-i},\end{equation} and such that the graded algebra associated
to this filtration $\gr' A$ is Loewy-graded. Then the Loewy
filtration on $A$  is equal to this given filtration, that is
$A^n=A_n$ for all $n=0,\dots ,m$.
\end{lema}
\pf Straightforward. \epf

We shall need the following important theorem due to Skryabin. The statement does not appear explicitly in \cite{Sk} but is contained in  the proof of \cite[Theorem 3.5]{Sk}. Let $H$ be a finite dimensional Hopf algebra.

\begin{teo}\label{freeness} If $A$ is a finite dimensional $H$-simple left $H$-comodule algebra and  $M \in {}^H\Mo_A$, then there exists $t\in \Na$ such that $M^t$
is a free $A$-module.\qed
\end{teo}

The following Lemma will be useful to distinguish equivalence classes of module categories. Let $\sigma: H\otimes H\to \ku$
be a Hopf 2-cocycle and $K$ be a left $H$-comodule algebra.
\begin{lema}\label{equivalencia-cocentral}
There is an equivalence of categories ${}^{H}\Mo_{K}\simeq
{}^{H^{\sigma}}\Mo_{K_{\sigma}}$.
In particular if $K\subseteq H$ is a left coideal
subalgebra, $Q=H/HK^+$ and $\sigma$ is
cocentral then the categories
${}^{H\!}\Mo_{K_{\sigma}}$, ${}^Q\Mo$ are equivalent.
\end{lema}
\pf See \cite[Lemma 2.1]{M2}. \epf

\section{Liftings of quantum linear spaces}\label{qls}

In this section we recall some results from \cite{AS1}. More
precisely, we shall recall the definition of a certain family of
finite-dimensional Hopf algebras such that the associated graded
Hopf algebras are the bosonization of a quantum linear space and a
group algebra of an Abelian group.

\subsection{Quantum linear spaces}

We shall use the notation from \cite{AS1}, \cite{AS2}. Let
$\theta\in \Na$ and $\Gamma$ be a finite Abelian group. A
\emph{datum for a quantum linear space} consists of elements $g_1,
\dots, g_{\theta}\in \Gamma$, $\chi_1,
\dots,\chi_{\theta}\in\widehat{\Gamma}$ such that
\begin{align}\label{qls1} q_i=\chi_i(g_i)\neq 1, \text{ for all }
i,\\ \label{qls2} \chi_i(g_j)\chi_j(g_i)=1,  \text{ for all }
i\neq j.
\end{align}

Let us denote $q_{ij}=\chi_j(g_i)$ and for any $i$ let $N_i>1$
denote the order of $q_i$. Denote $V=V(g_1,\dots,
g_{\theta},\chi_1, \dots,\chi_{\theta})$ the Yetter-Drinfeld
module over $\ku \Gamma$ generated as a vector space by
$x_1,\dots, x_{\theta}$ with structure given by
\begin{align} \delta(x_i)=g_i\ot x_i, \; h\cdot x_i=\chi_i(h)\,
x_i,
\end{align}
for all $i=1,\dots, \theta$, $h\in \Gamma$. The associated Nichols
algebra $\nic(V)$ is the graded braided Hopf algebra generated by
elements $x_1,\dots, x_{\theta}$ subject to relations
\begin{align}\label{qls3} x_i^{N_i}=0, \quad x_i x_j= q_{ij}\; x_j
x_i \; \text{ if }  i\neq j.
\end{align}
This algebra is called the \emph{quantum linear space} associated
to $V$, or to $(g_1,\dots, g_{\theta},\chi_1,
\dots,\chi_{\theta})$ and it is denoted by $\qs=\qs(g_1,\dots,
g_{\theta},\chi_1, \dots,\chi_{\theta})$. The gradation on
$\qs=\oplus_n\, \qs(n)$ is given as follows. If $n\in \Na$ then
$$ \qs(n)=<\{x^{r_1}_1\dots x^{r_{\theta}}_{\theta}: r_1+\dots+
r_{\theta}=n\}>_{\ku}$$

\begin{rmk}\label{whenN=2} The space $V$ decomposes as $V=\oplus^\theta_{i=1} V_{g_i}$, where $V_{g_i}=\{
v\in v: \delta(v)=g_i\ot v\}$ is the isotypic component of type
$g_i$. Since it can happen that for some $k\neq l$, $g_k=g_l$,
then $\dim V_{g_k}\geq 1$. If $\dim V_{g_k}>2$ for some $k$ then
there are at least two $g_i's$ equal to $g_k$. Using equation
\eqref{qls2} this implies that $q_k^2=1$, hence $N_k=2$.
\end{rmk}

\begin{rmk}\label{qcommutation} If $g, h\in \Gamma$ and $v\in
V_g$, $w\in V_h$ then there exists a scalar $q_{h,g}\in \ku$
that only depends on $g$ and $h$ such that $wv=q_{h,g}\, vw$.
Indeed it is enough to prove that if $x_i, x_k \in V_g$ and $x_j,
x_l\in V_h$ then $q_{ij}=q_{kl}$. Since $g=g_i=g_k$ then
$q_{ij}=q_{kj}$, and since $h=g_j=g_l$ then $q_{jk}=q_{lk}$. Using
that $q_{ij}q_{ji}=1$ we deduce that $q_{ij}=q_{kl}$.
\end{rmk}

Let us denote $U=\qs\#\ku\Gamma$ the Hopf algebra obtained by
bosonization. The Hopf algebra $U$ is coradically graded with
gradation given by $U=\oplus_n\, U(n)$, $U(n)=\qs(n)\# \ku\Gamma$,
see for example \cite[Lemma 3.4]{AS1}. Next we will describe a
family of Loewy-graded $U$-comodule algebras.

\begin{defi} If $W\subseteq V$ is a $\ku\Gamma$-subcomodule, we shall denote by
$\kc(W)$ the subalgebra of $\qs$ generated by elements $\{w: w\in
W\}$. Clearly $\kc(W)$ is a left coideal subalgebra of $U$.
\end{defi}

Let $F\subseteq \Gamma$ be a subgroup, $\psi\in
Z^2(F,\ku^{\times})$ a 2-cocycle and $W\subseteq V$  a
$\ku\Gamma$-subcomodule invariant under the action of $F$. Set
$\kc(W,\psi,F)=\kc(W)\otk\ku_{\psi} F$ with product and  left
$U$-comodule structure $\lambda:\kc(W,\psi,F)\to U\otk
\kc(W,\psi,F)$ given as follows. If $g\in G$, $w\in W_g$, $v,
v'\in W$ and $f, f'\in F$, then
$$ (v\ot f)(v'\ot f')=v f\cdot v' \ot \psi(f,f')\, ff',  $$
$$\lambda(w\ot f)= (w\#f)\ot  1\ot f + (1 \# gf)\ot w \ot f.$$
There is a natural inclusion of vector spaces $\kc(W,\psi,F)
\hookrightarrow U$. Using this inclusion the coaction $\lambda$
coincides with the coproduct of $U$. Clearly $\kc(W,1,F)$ is a coideal subalgebra of $U$. For any $n\in \Na$ set
$\kc(W,\psi,F)(n)=\kc(W,\psi,F)\cap U(n)$.

\begin{lema}\label{loewy-algebras} With the above given
gradation the algebra $\kc(W,\psi,F)$ is a Loewy-graded
$U$-comodule algebra.
\end{lema}
\pf Let be $x\in \kc(W,\psi,F)$, then $x=\sum_i\, x_i$, where each
$x_i\in U(i)$. Since the coproduct of $U$ coincides with the
coaction $\lambda$ and $U$ is a graded Hopf algebra, then
$$\lambda(x_i)\in \bigoplus^i_{j=0}\; U(j)\otk U(i-j)
\cap U\otk \kc(W,\psi,F).$$ Applying $\epsilon$ to the first
tensorand we obtain that $$x_i=( \epsilon\ot\id)\lambda(x_i)\in
U(i)\cap \kc(W,\psi,F),$$ thus for any $i$, $x_i\in
\kc(W,\psi,F)(i)$, hence we conclude that $\kc(W,\psi,F)=\oplus_i
\kc(W,\psi,F)(i)$. It is straightforward to prove that this
gradation is an algebra gradation and since $U$ is coradically
graded then $\kc(W,\psi,F)$ is a Loewy-graded $U$-comodule
algebra. \epf
\subsection{Liftings of quantum linear spaces}\label{liftings}

Given a datum for a quantum linear space $\qs=\qs(g_1,\dots,
g_{\theta},\chi_1, \dots,\chi_{\theta})$ for the group $\Gamma$, a
\emph{compatible datum} for $\qs$ and $\Gamma$ is a pair
$\Do=(\mu, \lambda)$ where $\mu=(\mu_i)$, $\mu_i\in \{0,1\}$ for
$i=1\dots\theta$ and $\lambda=(\lambda_{ij})$ where
$\lambda_{ij}\in \ku$ for $1\leq i<j \leq \theta$, satisfying
\begin{enumerate}
    \item $\mu_i$ is arbitrary if $g^{N_i}_i\neq 1$ and
    $\chi^{N_i}_i=1$, and $\mu_i=0$ otherwise.
    \item $\lambda_{ij}$ is arbitrary if $g_ig_j\neq 1$ and
    $\chi_i\chi_j=1$, and $0$ otherwise.
\end{enumerate}

The algebra $\Ac(\Gamma,\qs,\Do)$ is generated by $\Gamma$ and
elements $a_i$, $i=1\dots\theta$ subject to relations
\begin{equation}\label{lifting-qs1} g a_i =\chi_i(g)\,  a_ig,\quad a^{N_i}_i=
\mu_i (1-g^{N_i}_i), \; i=1\dots\theta,
\end{equation}
\begin{equation}\label{lifting-qs1} a_i a_j=\chi_j(g_i)\, a_j
a_i+\lambda_{ij}(1-g_ig_j), 1\leq i<j \leq \theta.
\end{equation}

The following result is \cite[Lemma 5.1, Thm. 5.5]{AS1}.

\begin{teo} The algebra $\Ac(\Gamma,\qs,\Do)$ has a Hopf algebra
structure with coproduct determined by
$$\Delta(g)=g\ot g,\quad \Delta(a_i)=a_i\ot 1+g_i\ot a_i, $$
for any $g\in \Gamma$, $i=1,\dots,\theta$. It is a pointed Hopf
algebra with coradical $\ku\Gamma$ and the associated graded Hopf
algebra with respect to the coradical filtration $\gr
\Ac(\Gamma,\qs,\Do)$ is isomorphic to $\qs\# \ku\Gamma$.\qed
\end{teo}

\section{Representations of tensor categories}\label{modc}

We shall recall the basic definitions of exact module categories
over a tensor category and we shall describe the strategy to
classify exact module categories over the tensor category of
representations of a finite-dimensional pointed Hopf algebra.

\subsection{Exact module categories} Given $\ca=(\ca, \ot, a,
\uno)$ a tensor category a \emph{module category} over $\ca$ is an
abelian category $\Mo$ equipped with an exact bifunctor $\otimes:
\ca \times \Mo \to \Mo$ and natural associativity and unit
isomorphisms $m_{X,Y,M}: (X\otimes Y)\otimes M \to X\otimes
(Y\otimes M)$, $\ell_M: \uno \otimes M\to M$ satisfying natural
associativity and unit axioms, see \cite{eo}, \cite{O1}. We shall
assume, as in \cite{eo}, that all module categories  have only
finitely many isomorphism classes of simple objects. \smallbreak

A module category is {\em indecomposable} if it is not equivalent
to a direct sum of two non trivial module categories. A module
category $\Mo$ over a finite tensor category $\ca$ is \emph{exact}
(\cite{eo}) if  for any projective $P\in \ca$ and any $M\in \Mo$,
the object $P\ot M$ is again projective in $\Mo$.

\subsection{Exact module categories over Hopf algebras} We shall
give a brief account of the results obtained in \cite{AM} on exact
module categories over the category $\Rep(H)$, where $H$ is a
finite-dimensional Hopf algebra.

\medbreak

If $\lambda:A\to H\otk A$ is a left $H$-comodule algebra the
category ${}_A\Mo$ is a module category over $\Rep(H)$. When $A$
is right $H$-simple then ${}_A\Mo$ is an indecomposable exact
module category \cite[Prop 1.20]{AM}. Moreover any module category
is of this form.

\begin{teo}\cite[Theorem 3.3]{AM}\label{mod-overhopf} If $\Mo$ is an exact idecomposable
module category over $\Rep(H)$ then $\Mo\simeq {}_A\Mo$ for some
right $H$-simple left comodule algebra $A$ with $A^{\co
H}=\ku$.\qed
\end{teo}

The proof of this result uses in a significant way the results of \cite{eo}, \cite{O1}. The main ingredient here is the $H$-simplicity of
the comodule algebra that helps in the classification results.

\smallbreak

Two left $H$-comodule algebras $A$ and $B$ are \emph{equivariantly Morita equivalent}, and we shall denote it by $A\sim_M B$, if
the module categories ${}_A\Mo$, ${}_B\Mo$ are equivalent as
module categories over $\Rep(H)$.

\begin{prop}\label{eq1}\cite[Prop. 1.24]{AM} The algebras $A$ and $B$ are
Morita equivariant equivalent if and only if there exists $P\in
{}^{H\!}\Mo_{B}$ such that $A\simeq \End_B(P)$ as $H$-comodule
algebras.\qed
\end{prop}

The left $H$-comodule structure on $\End_B(P)$ is given by
$\lambda:\End_B(P)\to H\otk \End_B(P)$, $\lambda(T)=T\_{-1}\ot
T\_0$ where
\begin{equation}\label{h-comod} \langle\alpha, T\_{-1}\rangle\,
T_0(p)=\langle\alpha, T(p\_0)\_{-1}\Ss^{-1}(p\_{-1})\rangle\,
T(p\_0)\_0,\end{equation} for any $\alpha\in H^*$,
$T\in\End_B(P)$, $p\in P$. It is easy to prove that
$\End_B(P)^{\co H}= \End^H_B(P)$.

\subsection{Exact module categories over pointed Hopf algebras}\label{techniq}
We shall explain the technique developed in \cite{M2} to compute
explicitly exact indecomposable module categories over some
families of pointed Hopf algebras.

\medbreak

Let $H$ be a finite-dimensional Hopf algebra. Denote by
$H_0\subseteq H_1 \subseteq\dots \subseteq H_m=H$ the coradical
filtration. Let us assume that $H_0=\ku\Gamma$, where $\Gamma$ is
a finite Abelian group, and that the associated graded Hopf
algebra $U=\gr H$. We shall further assume that $U=\nic(V)\# \ku
\Gamma$, where $V$ is a Yetter-Drinfeld module over $\ku \Gamma$
with coaction given by $\delta:V\to \ku\Gamma\otk V$, and
$\nic(V)$ is the Nichols algebra associated to $V$.

\medbreak

The technique presented in \cite{M2} to find all right $H$-simple
left $H$-comodule algebras is the following. Let $\lambda:A \to H\otk A$ be a right $H$-simple left
$H$-comodule algebra with trivial coinvariants. Consider the Loewy
filtration $A_0\subseteq\dots \subseteq A_m=A$ and the associated
right $U$-simple left $U$-comodule graded algebra $\gr A$.
\medbreak

There is an isomorphism  $\gr A\simeq \nic_A\# A_0$ of $U$-comodule algebras, where $\nic_A\subseteq \gr A$ is a certain $U$-subcomodule
algebra, and $A_0$ happens to be a right $H_0$-simple left $H_0$-comodule
algebra.

\medbreak

Since $H_0=\ku \Gamma$ then $A_0=\ku_{\psi} F$ where $F\subseteq
\Gamma$ is a subgroup and $\psi\in Z^2(F, \ku^{\times})$ is a
2-cocycle. The algebra $\nic_A$ can be seen as a subalgebra in
$\subseteq \nic(V)$ and under this identification $\nic_A$ is an homogeneous left $U$-coideal subalgebra. 

\medbreak

In conclusion, to determine all possible right $H$-simple
left $H$-comodule algebras we have to, first, find all homogeneous
left $U$-coideal subalgebras $K$ inside the Nichols algebra
$\nic(V)$, and then all \emph{liftings} of $K\#\ku_{\psi} F$, that is all
left $H$-comodule algebras $A$ such that $\gr A\simeq
K\#\ku_{\psi} F$.

\medbreak

The problem of finding coideal subalgebras can be a very difficult
one. Some work has been done in this direction for the small
quantum groups $u_q(\mathfrak{sl}_n)$ \cite{KL} and for $U^+_q(\mathfrak{so}_{2n+1})$ \cite{K} also very beautiful results are obtained for right coideal subalgebras inside Nichols algebras \cite{HS} and for the Borel part of a quantized enveloping algebra \cite{HK}.

\medbreak

The following result summarizes what we have explained before.

\begin{teo}\label{descomposition}
Under the above assumptions there exists a graded subalgebra
$\nic_A=\oplus_{i=0}^m \nic_A(i)\subseteq \nic(V)$, a subgroup
$F\subseteq \Gamma$, and a 2-cocycle $\psi\in Z^2(F,
\ku^{\times})$ such that
\begin{itemize}
\item[1.] $\nic_A(0)=\ku$, $\nic_A(i)\subseteq \nic^i(V)$ for all
$i=0,\dots, m$, \item[2.] $\nic_A(1)\subseteq V$ is a $\ku
\Gamma$-subcomodule stable under the action of $F$,
    \item[3.]  for any $n=1,\dots, m$, $\Delta(\nic_A(n))\subseteq
    \oplus_{i=0}^n \,U(i) \otk \nic_A(n-i)$,
    \item[4.] $\gr A\simeq \nic_A\#\, \ku_{\psi} F$ as left $U$-comodule algebras.
\end{itemize}
\end{teo}
\pf The proof of \cite[Proposition 7.3]{M2} extends \emph{mutatis
mutandis} to the case when the group $F$ is arbitrary. \epf

The algebra structure and the left $U$-comodule structure of
$\nic_A\#\, \ku_{\psi} F$ is given as follows. If $x,y\in \kc$,
$f,g\in F$ then
\begin{align*} (x\# g) (y\# f)=x( g\cdot y)\# \psi(g,f)\, gf,\\
\lambda(x\# g)=(x\_1 g )\ot (x\_2\# g),
\end{align*}
where the action of $F$ on $\nic_A$ is the restriction of the
action of $\Gamma$ on $\nic(V)$ as an object in
${}_{\Gamma}^{\Gamma}\YD$. Observe that if $\nic_A=\kc(W)$ for
some $\ku\Gamma$-subcomodule $W$ of $V$ invariant under the action
of $F$, then $\nic_A\#\, \ku_{\psi} F=\kc(W,\psi,F).$

\medbreak

Lemma \ref{lambda-on-V} below will be useful to find liftings of
comodule algebras over Hopf algebras coming from quantum linear
spaces.

\smallbreak

Let us further assume that there is a basis
$\{x_1,\dots,x_{\theta}\}$ of $V$ such that there are elements
$g_i\in\Gamma$ and characters $\chi_i\in \widehat{\Gamma}$,
$i=1,\dots,\theta$ such that $g\cdot x_i=\chi_i(g)\, x_i$,
$\delta(x_i)=g_i\ot x_i$ for all $i=1,\dots,\theta$.

\smallbreak

Let $(G,\lambda_0)$ be a Loewy-graded $U$-comodule algebra with
grading $G=\oplus_{i=0}^{m} G(i)$ such that $G(1)\simeq
V\#\ku_{\psi} F$ under the isomorphism in Theorem
\ref{descomposition} (4) and there is a subgroup $F\subseteq
\Gamma$ such that $G(0)\simeq\ku_{\psi} F$ as $U$-comodules, that
is there is a basis $\{e_f: f\in F\}$ of $G(0)$ such that
$$ e_f e_h=\psi(f,h)\; e_{fh},\quad \lambda_0(e_f)=f\ot e_f,$$
for all $f, h\in F$, and there are elements $y_i\in G(1)$ such
that $\lambda_0(y_i)=x_i\ot 1+ g_i\ot y_i$ for any
$i=1,\dots,\theta$. Also let $\lambda:A\to H\otk A$ be a left
$H$-comodule algebra such that $\gr A=G$.

\begin{lema}\label{lambda-on-V} Under the above assumptions
for any $i=1,\dots,\theta$ there are elements $v_i\in A_1$ such
that the class of $v_i$ in $A_1/A_0=G(1)$ is $\overline{v}_i=y_i$
and \begin{equation}\label{lambda-on-vi} \lambda(v_i)=x_i\ot 1+
g_i\ot v_i, \quad e_f v_i= \chi_i(f)\; v_ie_f,
\end{equation}
for any $i=1,\dots,\theta$ and any $f\in F$.
\end{lema}

\pf The existence of elements $v_i$ such that $\lambda(v_i)=x_i\ot
1+ g_i\ot v_i$ is \cite[Lemma 5.5]{M2}. For any $i=1,\dots,\theta$
and any $f\in F$ set
$$\Pc_{i,f}=\{y\in A_1: \lambda(y)=\mu\; fx_i\ot e_f + g_i\ot y,\; \;  \mu\in\ku\}.$$
The sets $\Pc_{i,f}$ are non-zero vector spaces since $e_fv_i\in
\Pc_{i,f}$, thus $\dim\Pc_{i,f}\geq 1.$ It is evident that if
$(i,f)\neq (i',f')$ then $\Pc_{i,f}\cap \Pc_{i',f'}=\{0\}$. Since
$\dim A_1=\dim G(0)+\dim G(1)=\mid F \mid(1+\theta)$ this forces
to $\dim\Pc_{i,f}=1$. Hence, since $e_f v_i e_{f^{-1}}, v_i\in
\Pc_{i,1}$ there exists $\nu\in \ku$ such that $e_f v_i
e_{f^{-1}}=\nu\; v_i$, but this scalar must be equal to
$\chi_i(f)$.\epf

\section{Module categories over quantum linear
spaces}\label{mod-overqs}

In this section we shall apply the technique explained above to
describe exact module categories over Hopf algebras coming from
quantum linear spaces. Let $\theta\in \Na$ and $\Gamma$ be a
finite Abelian group, $(g_1,\dots, g_{\theta},\chi_1,
\dots,\chi_{\theta})$ be a datum of a quantum linear space,
$V=V(g_1,\dots, g_{\theta},\chi_1, \dots,\chi_{\theta})$ the
Yetter-Drinfeld module over $\ku \Gamma$ and $\qs=\nic(V)$ its
Nichols algebra. Let $U= \qs\#\ku\Gamma$ denote the bosonization.
Elements in $U$ will be denoted by $v\# g$ instead of $v\ot g$ to
emphasize the presence of the semidirect product.

\medbreak

To describe all exact indecomposable module categories over
$\Rep(U)$ we will describe all possible right $U$-simple left
$U$-comodule algebras. For this description
 we shall need first the following crucial result which
essentially says that such comodule algebras are generated in
degree 1.
\begin{prop}\label{generation-in-degree-1} Let
$K=\oplus_{i=0}^m K(i)\in\qs$ be a graded subalgebra such that
\begin{itemize}
\item[1.] for all $i=0,\dots, m$,  $K(i)\subseteq \qs(i)$,
    \item[2.] $K(1)=W\subseteq V$ is a $\ku\Gamma$-subcomodule,
    \item[3.] $ \Delta(K(n))\subseteq \oplus_{i=0}^n\, U(i)\otk K(n-i).$
\end{itemize}
Then $K$ is generated as an algebra by $K(1)$, in another words
$K\simeq \kc(W)$.
\end{prop}
\pf Let $n\in \Na$, $0<n\leq m$. Since $W$ is a
$\ku\Gamma$-subcomodule then $W=\oplus^\theta_{i=0} W_{g_i}$. Let
$z\in K(n)$ be a nonzero element and let $1\leq d\leq \theta$ be
the number of $x_i's$ appearing in $z$ with non-zero coefficient.
We shall prove by induction on $n+d$ that $z\in\kc(W)$. If $n+d=2$
there is nothing to prove because in that case $d=1$ and $n=1$, so
assume that every time that $y\in K(n)$ is an element with $d$
different variables and $n+d< l$ then $y\in\kc(W)$. We shall use
the following claim as the main tool for the induction.
\begin{claim}\label{claim-inductive} If $2\leq n$ and $z\in K(n)$,
$z=\sum^{n-1}_{j=1} w^j y_j$, where $w\in W_h$, for some $h\in
\Gamma$ and for any $j=1,\dots,n-1$ the elements $y_j\in \qs(n-j)$
are such that the $x_i's$ appearing in the decomposition of $y_j$
does not appear in $w$. Then $y_j\in K(n-j)$ for any
$j=1,\dots,n-1$.
\end{claim}
\begin{proof}[Proof of Claim]
Let $z\in  K(n)$ such that $z=\sum^{n-1}_{j=1}\, w^j y_j$, as
above. Let $p: U\to U(1)$ be the linear map defined by: $p(wg)=wg$
for any $g\in \Gamma$ and $p(x)=0$ if $x\notin <wg: g\in
\Gamma>_{\ku}$.

Using (3) we obtain that $(p\ot\id)\Delta(z)\in  U(1)\ot K(n-1)$
and using that $\Delta(w)=w\ot 1+ h\ot w$, a simple computation
shows that
$$(p\ot\id)\Delta(z)=\sum^{n-1}_{j=1}\, w\beta_j\ot w^{j-1} y_j,$$
for some $\beta_j\in \ku\Gamma$. The second equality follows
because $p(y_j)=0$ for any $j=1,\dots,n-1$. Therefore the element
$\sum^{n-1}_{j=1}\, w^{j-1} y_j\in K(n-1)$. Repeating this process
we deduce that $y_{n-1}\in  K(n-1)$ and arguing inductively we
conclude that each $y_j \in K(n-j)$ for any $j=1,\dots,n-1$.
\end{proof}

Let $z\in K(n)$ be a nonzero element and let $1\leq d\leq \theta$
be the number of $x_i's$ appearing in $z$ with non-zero
coefficient. Assume that $n+d=l$. Since $\qs(n)$ is generated by
the monomials $\{x^{l_1}_1\dots x^{l_\theta}_\theta:
l_1+\dots+l_\theta=n\}$, and $K(n)\subseteq \qs(n)$ we can write
$z=\sum_{l_1+\dots+l_\theta=n} \, \alpha_{l_1,\dots,l_\theta} \;
x^{l_1}_1\dots x^{l_\theta}_\theta$ where $
\alpha_{l_1,\dots,l_\theta}\in \ku$ and $ 0\leq l_i\leq N_i$.
There is no harm to assume that the monomial $x_1$ appears, that
is there exists $l_1,\dots,l_\theta$ with $0<l_1$ such that
$\alpha_{l_1,\dots,l_\theta}\neq 0$, since otherwise we can repeat
the argument with $x_2$ or $x_3$ and so on.

\smallbreak

Under this mild assumption the space $W_{g_1}$ is not zero.
Moreover there is an element $0\neq w\in W_{g_1}$ where
$w=\sum^\theta_{i=1}\, a_i\, x_i$ and $a_1\neq 0$.  Indeed, if
$\pi:U\to V_{g_1}$ denotes the canonical projection, the quantum
binomial formula implies that for any $j=1,\dots,\theta$
\begin{align}\label{qbinom}\Delta
(x_{j}^{l_{j}}) = \sum^{l_j}_{k_j=0}
{\binom{l_{j}}{k_{j}}}_{q_{j}} \;x_{j}^{l_j-k_{j}} g_j^{k_{j}}
\otimes x_{j}^{k_{j}},
\end{align}where
${\binom{l_{j}}{i_{j}}}_{q_{j}}$ denotes the quantum Gaussian
coefficients. Using (3) we know that $ (\id\ot\pi)\Delta(z)\in
H(n-1)\otk W_{g_1}$ and equals to
$$\sum_{\substack{j=1,\dots,\theta\\l_1+\dots+l_\theta=m}}
\;\,\alpha_{l_1,\dots,l_\theta}{\binom{l_{j}}{1}}_{q_{j}} \;
x^{l_1}_1\dots x^{l_j-1}_j g_j\dots x^{l_\theta}_\theta \ot x_j.$$
Since there exists $l_1,\dots, l_\theta$ such that
$l_1+\dots+l_\theta=m$ and $0<l_1$,
$\alpha_{l_1,\dots,l_\theta}\neq 0$ then
$(\id\ot\pi)\Delta(z)=\sum h_j\ot w_j$ where at least one $w_j\neq
0$ written in the basis $\{x_1, \dots, x_\theta\}$ has positive
coefficient in $x_1$.

\smallbreak

Up to reordering the variables we can assume that
$g_1=g_2=\dots=g_{r_1}$ and if $r_1< j$ then $g_j\neq g_1$. In
this case $\dim V_{g_1}=r_1$. We shall treat separately the
following three cases: Case (A) $r_1=1$, Case (B) $r_1=2$, Case
(C) $r_1> 2$.

\smallbreak

Since $W\subseteq V$ is a $\ku\Gamma$-subcomodule, in case (A)
$W_{g_1}=\{0\}$ or $W_{g_1}=V_{g_1}$. We have proven that
$W_{g_1}$ is not zero, hence $W_{g_1}=V_{g_1}$. Let us write
$z=\sum^{N_1}_{i=0} \, x^i_1 y_i$, where $y_i\in \qs(n-i)$, and
$x_1$ does not appear in any factor of $y_i$, that is, for any
$i=0,\dots, N_1$
$$y_i=\sum \gamma^i_{l_2,\dots,l_\theta}\; x^{l_2}_2\dots
x^{l_\theta}_\theta, $$ for some $\gamma^i_{l_2,\dots,l_\theta}\in
\ku$. The projection $V\to V$ that maps $V_{g_1}$ to zero, extends
to an algebra map $q:\qs\to \qs$. Using (3) we get that $(q\ot
q)\Delta(z)= \Delta(y_0)$, thus $y_0\in K(n)$, and therefore
$z-y_0=\sum^{N_1}_{i=1} \, x^i_1 y_i\in K(n)$. Using Claim
\ref{claim-inductive} we deduce that for any $i=1,\dots, N_1$ the
element $y_i\in K(n-i)$ and by inductive hypothesis each $y_i\in
\kc(W)$ for all $i=0,\dots, N_1$.

\medbreak

Now we proceed to the case (B). In this case  $V_{g_1}$ has basis
$\{x_1, x_2\}$, and $W_{g_1}=0$, $\dim W_{g_1}=1$ or
$W_{g_1}=V_{g_1}$. The first case is impossible. If
$W_{g_1}=V_{g_1}$ then $x_1\in W_{g_1}$ and we proceed as in case
(A). Let us assume that $\dim W_{g_1}=1$, that is $ W_{g_1}$ is
generated by an element $a\, x_1+ b\, x_2$, for some $a, b\in
\ku$, where we can assume that $b\neq 0$ because otherwise $x_1
\in W_{g_1}$.

Let $p:K\to V_{g_1}$ be the canonical projection. Follows from
\eqref{qbinom} that $(\id\ot p)\Delta(z)$ equals to
\begin{align*}\sum_{l_1+\dots+l_\theta=n} \, \alpha_{l_1,\dots,l_\theta}{\binom{l_{1}}{1}}_{q_{1}} \;
&x^{l_1-1}_1g_1 x^{l_2}_2\dots x^{l_\theta}_\theta \ot x_1 +\\
&+ \sum_{l_1+\dots+l_\theta=n} \,
\alpha_{l_1,\dots,l_\theta}{\binom{l_{2}}{1}}_{q_{2}} \; x^{l_1}_1
x^{l_2-1}_2g_2 x^{l_3}_3\dots x^{l_\theta}_\theta \ot x_2.
\end{align*}
Using (3) we obtain that $(\id\ot p)\Delta(z)\in H(n-1)\ot
W_{g_1}$, hence there exists an element $v\in  H(n-1)$ such that
$(\id\ot p)\Delta(z)=a v\ot x_1 + b v\ot x_2$, thus
$$av=\sum_{l_1+\dots+l_\theta=n} \, \alpha_{l_1,\dots,l_\theta}{\binom{l_{1}}{1}}_{q_{1}} \;
x^{l_1-1}_1g_1 x^{l_2}_2\dots x^{l_\theta}_\theta,$$
$$bv=\sum_{l_1+\dots+l_\theta=n} \, \alpha_{l_1,\dots,l_\theta}{\binom{l_{2}}{1}}_{q_{2}} \; x^{l_1}_1 x^{l_2-1}_2g_2 x^{l_3}_3\dots x^{l_\theta}_\theta.$$
Comparing coefficients from the above equations and using that
$g_1=g_2$, $x_2x_1=q_1\; x_1x_2$, $g_1x_1=q_1\; x_1g_1$,
$g_1x_2=q^{-1}_1\; x_2g_1$ we obtain that
\begin{equation}\label{alpha-coef} \alpha_{l_1+1, l_2,\dots,l_\theta}=\frac{a}{b}\;
\frac{q^{l_2+1}_1-1}{q^{l_1+1}_1-1}\; \alpha_{l_1,
l_2+1,l_3,\dots,l_\theta}.
\end{equation}
For any $m\in \Na$ set
$\gamma^m_{l_3,\dots,l_\theta}=\alpha_{0,m,l_3,\dots,l_\theta}$,
then if $l_1+l_2=m$ we deduce from equation \eqref{alpha-coef}
that
\begin{equation}\label{alpha-coef2}\alpha_{l_1, l_2,\dots,l_\theta}= {\binom{m}{l_1}}_{q_{1}}
a^{l_1} b^{m-l_1}
\frac{\gamma^m_{l_3,\dots,l_\theta}}{b^m}.\end{equation}

Then we can write the element $z$ as
\begin{align*}&\sum_{m\geq
0} \sum_{\substack{l_3+\dots+l_\theta=n-m\\ l_1+l_2=m} } \,
\alpha_{l_1,\dots,l_\theta} \; x^{l_1}_1\dots
x^{l_\theta}_\theta=\\
&=\sum_{m\geq 1 } \sum_{\substack{l_1+l_2=m \\
l_3+\dots+l_\theta=n-m}} \alpha_{l_1,\dots,l_\theta} \;
x^{l_1}_1\dots x^{l_\theta}_\theta +
\sum_{l_3+\dots+l_\theta=n}\alpha_{0,0,l_3,\dots,l_\theta} \;
x^{l_3}_3\dots x^{l_\theta}_\theta.
\end{align*}
Using the same argument as before and the inductive hypothesis we
deduce that the element
$y_0=\sum_{l_3+\dots+l_\theta=n}\alpha_{0,0,l_3,\dots,l_\theta} \;
x^{l_3}_3\dots x^{l_\theta}_\theta \in \kc(W)$ and $z-y_0\in
K(n)$. From \eqref{alpha-coef2} we conclude that
$$z-y_0= \sum_{m\geq 1 } \sum_{
l_3+\dots+l_\theta=n-m}\;
\frac{\gamma^m_{l_3,\dots,l_\theta}}{b^m} (ax_1+bx_2)^m x_3^{l_3}
\dots x^{l_\theta}_\theta,
$$
hence by Claim \eqref{claim-inductive} $z-y_0 \in \kc(W)$ thus
$z\in \kc(W)$. Case (C) can be treated in a similar way as case
(B). \epf

\begin{rmk} The above result uses in an essential way the
 structure of $\qs$ and it is no longer true for
arbitrary Nichols algebras. It is worth to mention that this is
one of the main difficulties to classify module categories over,
for example, $\Rep(u_q(\mathfrak{sl}_3))$, since there are homogeneous coideal subalgebras that are not generated in degree 1.
\end{rmk}

Let us define now a family of right $U$-simple left $U$-comodule
algebras. Let $F\subseteq \Gamma$ be a subgroup, $\psi\in Z^2(F,
\ku^{\times})$ a 2-cocycle and $\xi=(\xi_{i})_{i=1\dots \theta}$,
$\alpha=(\alpha_{ij})_{1\leq i<j\leq\theta}$ be two families of
elements in $\ku$ satisfying
\begin{align}\label{parameters1} \xi_{i}=0\; \text{ if }\;
g^{N_i}_i\notin F  \text{ or }\;\chi^{N_i}_i(f)\neq
\psi_{g^{N_i}_i}(f),
\end{align}
\begin{align}\label{parameters2}
 \alpha_{ij}=0\; \text{ if }\;
g_ig_j \notin F  \text{ or }\;
 \chi_i\chi_j(f)\neq \psi_{g_ig_j}(f),
\end{align}
for all $f\in F$. In this case we shall say that the pair $(\xi,
\alpha)$ is \emph{compatible comodule algebra datum} with respect
to the quantum linear space $\qs$, the 2-cocycle $\psi$ and the
group $F$.

\begin{defi} The algebra $\Ac(V, F, \psi, \xi, \alpha)$ is the algebra
generated by elements in $\{v_i:i=1\dots \theta\}$, $\{e_f: f\in
F\}$ subject to relations
\begin{align}\label{relations1} e_f e_g = \psi(f,g)\; e_{fg},\quad
e_f v_i =  \chi_i(f)\; v_i e_f,
\end{align}
\begin{align}\label{relations2}  v_i v_j - q_{ij}\; v_jv_i=\begin{cases}
  \alpha_{ij}\; e_{ g_ig_j}\;\; \;\text{ if }
  g_ig_j\in F\\
0 \;\;\; \;\;\;\quad\quad\text{ otherwise, }
\end{cases}
\end{align}
\begin{align}\label{relations3} v_i^{N_i}=\begin{cases}  \xi_{i}\; e_{ g_i^{N_i}}\;\; \;\text{ if } g_i^{N_i}\in F\\
0 \;\;\; \;\;\;\quad\quad\text{ otherwise, }
\end{cases}
\end{align}

for any $1\leq i< j\leq \theta$. Observe that we are abusing of
the notation since we are changing the name of the variables $x_i$
by $v_i$ to emphasize that the elements no longer belong to $U$.
If $W\subseteq V$ is a $\ku\Gamma$-subcomodule invariant under the
action of $F$, we define $\Ac(W, F, \psi, \xi, \alpha)$ as the
subalgebra of $\Ac(V, F, \psi, \xi, \alpha)$ generated by $W$ and
$\{e_f: f\in F\}$.

\end{defi}

The algebra $\Ac(V, F, \psi, \xi, \alpha)$ is a left $U$-comodule
algebra with structure map $ \lambda:\Ac(V, \psi, \xi, \alpha)\to
U\otk \Ac(V,F, \psi, \xi, \alpha)$ given by:
$$ \lambda(v_i)=x_i\ot 1 + g_i\ot v_i, \quad  \lambda(e_f)=f\ot e_f,$$
It is clear that the map $\lambda$ is well defined and is an
algebra morphism and that the subalgebra $\Ac(W, F, \psi, \xi,
\alpha)$ is a $U$-subcomodule.

\begin{rmk}\begin{itemize}
    \item[1.] The algebra $\Ac(W, F, \psi, \xi, \alpha)$ does not
depend on the class of the 2-cocycle $\psi$.
    \item[2.] If $W=0$ then $\Ac(W, F, \psi, \xi, \alpha)= \ku_{\psi} F.$
\end{itemize}
\end{rmk}

\begin{prop}\label{propertiesA1} Under the above assumptions the
 following assertions hold.
\begin{enumerate}

    \item For any 2-cocycle $\psi$ of $\Gamma$ and any
    compatible comodule algebra
    datum $(\xi,\alpha)$ the
    algebra $\Ac(V, \Gamma, \psi, \xi, \alpha)$ is a Hopf-Galois extension
    over the field $\ku$.

\item The Loewy filtration on $\Ac=\Ac(V, F, \psi, \xi, \alpha)$
is given as follows
\begin{align}\label{l-filt} \Ac_n=<\{ e_f v^{r_1}_1 \dots v^{r_{\theta}}_{\theta}:
r_1+\dots +r_{\theta}=m:m\leq n , f\in F\}>_{\ku}.
\end{align}

    \item The graded algebra $\gr \Ac(W, F, \psi, \xi, \alpha)$ is
    isomorphic to  $\kc(W,\psi,F)$.
\end{enumerate}
\end{prop}

\pf The proof of (1) is standard. One must show that the canonical
map $\beta: \Ac(V, G, \psi, \xi, \alpha)\otk \Ac(V, G, \psi, \xi,
\alpha)\to U\otk \Ac(V, G, \psi, \xi, \alpha)$, $\beta(a\ot
b)=a\_{-1}\ot a\_0 b$ is bijective. For this it is enough to show
that the elements $g\ot 1$ and $x_i\ot 1$ are in the image of
$\beta$ for all $g\in G$, $i=1,\dots ,\theta$, and this follows
because $\beta(e_g\ot e_{g^{-1}})=g\ot 1$, $\beta(v_i\ot 1- e_{
h_i}\ot e_{ h^{-1}_i}v_i)=x_i\ot 1$.

\smallbreak

 The
filtration on $\Ac=\Ac(V, F, \psi, \xi, \alpha)$ defined by
\eqref{l-filt} satisfies the hypothesis in Lemma \ref{loewy=grad},
hence it coincides with the Loewy filtration. This proves (2).
\smallbreak

The algebra $\gr \Ac(W, F, \psi, \xi, \alpha)$ is a Loewy-graded
$U$-comodule algebra satisfying
$$\gr \Ac(W, F, \psi, \xi, \alpha)(0)=\ku_{\psi} F,\;\;\;
\gr \Ac(W, F, \psi, \xi, \alpha)(1)= W\otk \ku F.$$ Thus (3)
follows from Theorem \ref{descomposition} and Proposition
\ref{generation-in-degree-1}. \epf

Now we can state the main result of this section.

\begin{teo}\label{clasification-qspaces} Let
$\theta\in \Na$,  $\Gamma$ be a finite Abelian group, $g_1,
\dots, g_{\theta}\in \Gamma$, $\chi_1,
\dots,\chi_{\theta}\in\widehat{\Gamma}$ be a datum for a quantum linear space, with associated Yetter-Drinfeld
module over $\ku \Gamma$ $V=V(g_1,\dots,
g_{\theta},\chi_1, \dots,\chi_{\theta})$ and $U=\nic(V)\# \ku\Gamma$.
  
 If $\Mo$ is an exact indecomposable module category
over $\Rep(U)$ then there exists a subgroup $F\subseteq \Gamma$, a
2-cocycle $\psi\in Z^2(F, \ku^{\times})$, a compatible datum
$(\xi,\alpha)$ and $W\subseteq V$ a subcomodule invariant under
the action of $F$ such that $\Mo\simeq {}_{\Ac(W,F, \psi, \xi,
\alpha)}\Mo$ as module categories.

\end{teo}

\pf By Theorem \ref{mod-overhopf} there exists a right $U$-simple
left $U$-comodule algebra $\lambda:\Ac \to H\otk \Ac$ with trivial
coinvariants such that $\Mo\simeq {}_{\Ac}\Mo$ as module
categories over $\Rep(U)$. Since $U_0=\ku \Gamma$, and $\Ac_0$ is
a right $U_0$-simple left $U_0$-comodule algebra \cite[Proposition
4.4]{M2} then $\Ac_0=\ku_{\psi} F$ for some subgroup $F\subseteq
G$ and a 2-cocycle $\psi\in Z^2(F, \ku^{\times})$. Thus we may
assume that $\Ac\neq \Ac_0$. By Theorem \ref{descomposition} there
exists an homogeneous coideal subalgebra $\nic_A\subseteq \qs$
such that
 $\gr \Ac\simeq \nic_A\# \ku_{\psi} F$. Proposition
\ref{generation-in-degree-1} implies that $\nic_A=\kc(W)$ for some
$\ku\Gamma$-subcomodule $W\subseteq V$ invariant under the action
of $F$, thus $\gr \Ac\simeq \kc(W,\psi,F)$. Since $\Ac\neq \Ac_0$
the space $W$ is not zero.

\smallbreak

Let us assume first that $\gr \Ac\simeq \kc(V,\psi,F)$. By Lemma
\ref{lambda-on-V}, there are elements $\{v_i: i=1\dots\theta\}$ in
$\Ac$ such that for all $f\in F$
$$\lambda(v_i)=x_i\ot 1+ g_i\ot v_i,\quad e_f v_i =  \chi_i(f)\; v_i e_f.$$
Since $\gr \Ac$ is generated as an algebra by $V$ and $\ku_{\psi}
F$ then $\Ac$ is generated as an algebra by the elements $\{v_i:
i=1\dots\theta\}$  and $\ku_{\psi} F$. Since $x_i\ot 1$ and
$g_i\ot v_i$ $q_i$-commute then the quantum binomial formula
implies that
$$\lambda(v^{N_i}_i)=g^{N_i}_i\ot v^{N_i}_i,$$
thus $v^{N_i}_i\in \Ac_0$ and there exists $\xi_i\in \ku$ such
that $v^{N_i}_i=\xi_i\, e_{g^{N_i}_i}$ if $g^{N_i}_i\in F$,
otherwise $v^{N_i}_i=0.$ If $i\neq j$ then
$\lambda(v_iv_j-q_{ij}\; v_j v_i)=g_i g_j\ot (v_iv_j-q_{ij}\; v_j
v_i).$ Hence $v_iv_j-q_{ij}\; v_j v_i\in \Ac_0$, and therefore
$v_iv_j-q_{ij}\; v_j v_i=\sum_{f\in F}\; \zeta_f\, e_f.$ Thus we
conclude that if $g_i g_j\in F$ then there exists $\alpha_{ij}\in
\ku$ such $v_iv_j-q_{ij}\; v_j v_i=\alpha_{ij}\;e_{g_i g_j}$, and
if $g_i g_j\notin F$ then $v_iv_j-q_{ij}\; v_j v_i=0$. It is clear
that $(\xi, \alpha)$ is compatible with the quantum linear space
and $\psi$, therefore there is a projection $\Ac(V, G, \psi, \xi,
\alpha)\twoheadrightarrow \Ac$ of $U$-comodule algebras, but both
algebras have the same dimension, since by Proposition
\ref{propertiesA1} (3) $\gr \Ac\simeq \gr \Ac(V, G, \psi, \xi,
\alpha)$, thus they are isomorphic.

\medbreak

If $\gr \Ac\simeq \kc(W,\psi,F)$ for some $\ku\Gamma$-subcomodule
$W\subseteq V$ invariant under the action of $F$ we proceed as
follows. We shall define an $U$-comodule algebra $\Do$ such that
$\gr \Do= \kc(V,\psi,F)$ such that $\Ac$ is a $U$-subcomodule
algebra of $\Do$, and this will finish the proof of the theorem
since $\Do\simeq \Ac(V,F, \psi, \xi, \alpha)$ and by definition
$\Ac(W,F, \psi, \xi, \alpha)$ is the subcomodule algebra of
$\Ac(V,F, \psi, \xi, \alpha)$ generated by $W$ and $\ku F$.

\medbreak

Using again Lemma \ref{lambda-on-V} there is an injective map
$W\hookrightarrow \Ac_1$ such that for any $h\in \Gamma$, $w\in
W_h$
$$ \lambda(w)=w\#1 \ot 1 + 1\#h \ot w.$$
Observe that here we are abusing of the notation since the element
$w$ also denotes the element in $\Ac_1$ under the above inclusion.
Using this identification $\Ac$ is generated as an algebra by $W$
and $\ku F$.

 Let $W'\subseteq V$
be a $\ku\Gamma$-subcomodule and an $F$-submodule such that
$V=W'\oplus W$. Set $\Do= \kc(W')\otk \Ac$, with algebra structure
determined by
$$ (1\ot a)(1\ot b)=1\ot ab,\;\; (x\ot 1)(y\ot 1)=xy\ot 1,
\;\;(x\ot 1)(1\ot a)= x\ot a, $$
$$(1\ot e_f)(v\ot 1)=f\cdot v\ot e_f,\;\; (1\ot w)(v\ot 1)=
q_{h,g} (v\ot w), $$ for any $a, b\in \Ac$, $x,y\in \kc(W')$,
$f\in F$, $h, g\in \Gamma$, $w\in W'_h$, $v\in W_g$. Here the
scalar $q_{h,g}\in \ku$ is determined by the equation in $U$: $wv=
q_{h,g} \, vw$, see Remark \ref{qcommutation}. Let us define
$\widetilde{\lambda}:\Do\to U\ot_{\ku} \Do$ the coaction by:
$$\widetilde{\lambda}(x\ot a)= x\_{-1}a\_{-1}\ot x\_0\ot a\_0,$$
for all $a\in \Ac$, $x\in \kc(W')$. By a direct computation one
can see that $\widetilde{\lambda}$ is an algebra map. It is not
difficult to see that $\Do_0=\ku_{\psi} F$ and that $\gr
\Do(1)=V\otk \ku F$, thus $\gr \Do\simeq \kc(V,\psi,F)$.\epf

\begin{exa} This example is a particular case of a classification
result obtained in \cite[\S 4.2]{eo} for the representation category of finite supergroups.

Let $\theta\in \Na$, $\Gamma$ be an Abelian group and $u\in
\Gamma$ be an element of order 2. Set $g_1=\dots=g_{\theta}=u$ and
for any $i=1,\dots,\theta$ let $\chi_i\in \widehat{\Gamma}$ be 
characters such that $\chi_i(u)=-1$. If $V=V(g_1,\dots,
g_{\theta},\chi_1, \dots,\chi_{\theta})$ then the associated
quantum linear space is the exterior algebra $\wedge V$. In this case $U=\wedge V\#  \ku \Gamma$.

\medbreak

Let $F\subseteq \Gamma$ be a subgroup and $\psi\in
Z^2(F,\ku^{\times})$ a 2-cocycle. A compatible comodule algebra
datum in this case is a pair $(\xi,\alpha)$
 satisfying
\begin{align}\label{sym-form} \xi_{i}=0\; \text{ if }\;
\chi^{2}_i(f)\neq 1,\quad \alpha_{ij}=0\; \text{ if }\;
 \chi_i\chi_j(f)\neq 1,\;\; \text{for all}\; f\in F.
\end{align}
If $W\subseteq V$ is a subspace stable under the action of $F$ the
algebra $\Ac(W, F, \psi, \xi, \alpha)$ is isomorphic to the
semidirect product $Cl(W, \beta)\# \ku_{\psi} F$ where $Cl(W,
\beta)$ is the Clifford algebra  associated to the symmetric
bilinear form $\beta:V\times V\to \ku$ invariant under $F$ defined
by
\begin{align}\label{sym-form2}\beta(v_i, v_j)=\begin{cases}\frac{\alpha_{ij}}{2} \;\;\; \text{ if } i\neq j\\
\xi_i \;\;\;\;\;\; \text{ if } i=j .\end{cases}\end{align} Reciprocally, if
$W\subseteq V$ is a $F$-submodule, any symmetric bilinear form
$\beta:W\times W\to \ku$ invariant under $F$ defines a comodule
algebra datum $(\xi,\alpha)$. Indeed, take $U\subseteq V$ a
$F$-submodule such that $V=W\oplus U$ and define
$\widehat{\beta}:V\times V\to \ku$ such that $\widehat{\beta}(w_1,
w_2)=\beta(w_1, w_2)$ if $w_1, w_2\in W$ and $\widehat{\beta}(v,
u)=0$ for any $v\in V$, $u\in U$. Follows that the pair
$(\xi,\alpha)$ defined by equation \eqref{sym-form2} using
$\widehat{\beta}$ gives a compatible comodule algebra datum.

\end{exa}

\begin{rmk} It would be interesting to give an explicit
description of data $(W,F, \psi, \xi, \alpha)$ such  that  the
algebra $\Ac(W,F, \psi, \xi, \alpha)$ is simple. This would  give
a description  of  twists  over  $U$, \emph{i.e.} fiber functors
for $\Rep(U)$.
\end{rmk}

\subsection{Equivariant  equivalence classes of algebras $\Ac(W,F, \psi, \xi, \alpha)$}\label{moritaeq}

In this section we shall distinguish equivalence classes of module categories of Theorem \ref{clasification-qspaces}, that is equivariant Morita equivalence classes of the algebras $\Ac(W,F, \psi, \xi, \alpha)$.

\smallbreak

Let $U$ be the Hopf algebra as in the previous section.
Let $W, W'\subseteq V$ be subcomodules, $F, F'\subseteq \Gamma$ be two subgroups, $\psi\in Z^2(F,
\ku^{\times})$, $\psi'\in Z^2(F',
\ku^{\times})$  2-cocycles and $(\xi, \alpha)$, $(\xi',
\alpha')$ compatible comodule algebra datum with respect
to the quantum linear space $\qs$, the 2-cocycles $\psi$, $\psi'$ and the groups $F, F'$ respectively.

\begin{teo}\label{morita-equivariant} The associated right simple left  $U$-comodule algebras to these data $\Ac(W,F, \psi, \xi, \alpha)$, $\Ac(W',F', \psi', \xi', \alpha')$ are equivariantly Morita equivalent if and only if $(W,F, \psi, \xi, \alpha)=(W',F', \psi', \xi', \alpha')$.
\end{teo}

We shall need first the following result.

\begin{lema}\label{iso-alg} The algebras $\Ac(W,F, \psi, \xi, \alpha)$, $\Ac(W',F', \psi', \xi', \alpha')$ are isomorphic as left $U$-comodule algebras if and only if $W=W'$, $F=F'$,  $\psi=\psi'$, $\xi=\xi'$ and $\alpha=\alpha'$.
\end{lema}
\pf Let $\Phi: \Ac(W,F, \psi, \xi, \alpha)\to \Ac(W',F', \psi', \xi', \alpha')$ be an isomorphism of $U$-comodule algebras. The map $\Phi$ induces an isomorphism between $\ku_\psi F$ and $\ku_{\psi'} F'$ that must be the identity, thus $F$ is equal to $F'$ and $\psi=\psi'$ in $H^2(F, \ku^{\times})$.

\medbreak

Let $\widetilde{W}\in {}^{\ku\Gamma}_{\ku F}\Mo$ be a complement of $W$ in $V$, that is $V=W\oplus \widetilde{W}$. Let us define a map $\widetilde{\Phi}:\Ac(V,F, \psi, \xi, \alpha)\to \Ac(V,F, \psi, \xi', \alpha')$ such that $\widetilde{\Phi}(a)=\Phi(a)$ whenever $a\in \Ac(W,F, \psi, \xi, \alpha)$.

It is enough to define $\widetilde{\Phi}$ on $V$ and $\{e_f: f\in F\}$ since $\Ac(V,F, \psi, \xi, \alpha)$ is generated as an algebra by these elements. Set
$$\widetilde{\Phi}(w)=\Phi(w), \quad  \widetilde{\Phi}(u)= u,\quad \widetilde{\Phi}(e_f)=\Phi(e_f),$$
for any $w\in W$, $u\in \widetilde{W}$, $f\in  F$. It is straightforward to prove that $\widetilde{\Phi}$ is an $U$-comodule algebra map, and necessarily  $\widetilde{\Phi}$ is the identity map, whence $\Phi$ is the identity and the Lemma follows.
\epf

\pf[Proof of Theorem \ref{morita-equivariant}] Let us assume that  $\Ac=\Ac(W,F, \psi, \xi, \alpha)$ and $\Ac'=\Ac(W',F', \psi', \xi', \alpha')$ are equivariantly Morita equivalent. Thus there exists an equivariant Morita context $(P,Q,f,g)$, see \cite{AM}. That is $P\in {}_{\Ac'}^U\Mo_{\Ac}$, $Q\in {}_{\Ac}^U\Mo_{\Ac'}$ and $f:P\ot _{\Ac} Q\to \Ac'$, $g:Q\ot_{\Ac'} P\to \Ac$ are bimodule isomorphisms and $\Ac'\simeq \End_{\Ac}(P)$ as comodule algebras, where the comodule structure on $\End_{\Ac}(P)$ is given in \eqref{h-comod}.

\medbreak

Let us denote by $\delta:P\to U\otk P$ the coaction. Consider the filtration on $P$ given by $P_i=\delta^{-1}(U_i\otk P)$ for any $i=0\dots m$. This filtration is compatible with the Loewy filtration on $\Ac$, that is $P_i\cdot \Ac_j\subseteq P_{i+j}$ for any $i,j$ and for any $n=0\dots m$, $\delta(P_n)\subseteq \sum_{i=0}^n \, U_i\otk P_{n-i}$.

The space $P_0\cdot \Ac$ is a subobject of $P$ in the category ${}^U\Mo_{\Ac}$, thus we can consider the quotient $\overline{P}= P/ P_0\cdot \Ac$. Let us denote by $\overline{\delta}$ the coaction of $\overline{P}$. Clearly $\overline{P}_0=0$, therefore $\overline{P}=0$. Indeed, if $\overline{P}\neq 0$ there exists an element $q\in \overline{P}_n$ such that $q\notin \overline{P}_{n-1}$, but $\overline{\delta}(q)\subseteq \sum_{i=0}^n \, U_i\otk \overline{P}_{n-i}$. Since $\overline{P}_0=0$ then $\overline{\delta}(q)\in U_{n-1}\otk \overline{P}$ which contradicts the assumption. Hence $P=P_0\cdot \Ac$.

\medbreak

Since $P_0\in  {}^{\ku \Gamma}\Mo_{\ku_{\psi} F}$ then by Lemma \ref{equivalencia-cocentral} there exists an object $N\in {}^C\Mo$, $C=\ku \Gamma/ \ku \Gamma(\ku F)^+$ such that   $P_0\simeq N\otk \ku_{\psi} F$ as objects in ${}^{\ku \Gamma}\Mo_{\ku_{\psi} F}$.
The right $\ku_{\psi} F$-module structure on $N\otk \ku_{\psi} F$ is the regular action on the second tensorand and the left $\ku \Gamma$-comodule structure is given by $\delta:N\otk \ku_{\psi} F\to \ku \Gamma \otk N\otk \ku_{\psi} F$, $\delta(v\ot e_f)=v\_{-1}f\ot v\_0 \ot e_f$, $v\in N$, $f\in F$. Here we are identifying the quotient $C$ with $\ku \Gamma/F$.
Observe that $P= (N\ot 1)\cdot \Ac$. It is not difficult to prove that the action $(N\ot 1)\ot \Ac\to P$ is injective, thus $\dim P= \dim N\, \dim \Ac$. In a similar way one may prove that $\dim Q= s\,\dim \Ac'$ for some $s\in \Na$.

\medbreak

If $\dim N=1$ then there exists an element $g\in \Gamma$ and a non-zero element $v$ such that $\delta(v)=g\ot v$ and $P\simeq v\cdot\Ac$, where the left $U$-comodule structure is given by $\delta(v\cdot a)=ga\_{-1}\ot v\cdot a\_0$, for all $a\in\Ac$. In this case the map  $\varphi:g\Ac g^{-1}\to \End_{\Ac}(P)$ given by
$$ \varphi(gag^{-1})(v\cdot b)=v\cdot ab,$$
for all $a, b\in \Ac$ is an isomorphism of $U$-comodule algebras. Hence $\Ac'\simeq \Ac$. Thus, the proof of the Theorem follows from Lemma \ref{iso-alg} once we prove that $\dim N=1$.

Using Theorem \ref{freeness} there exists $t, s\in\Na$ such that $P^t$ is a free right $\Ac$-module, \emph{i.e}. there is a  vector space $M$ such that $P^t\simeq M\otk \Ac$, hence
\begin{align}\label{dimension1} t\, \dim N=\dim M.
\end{align}
Since $P\ot _{\Ac} Q\simeq \Ac'$ then $P^t\ot _{\Ac} Q\simeq M\otk Q\simeq \Ac'^t$, then $\dim M\dim Q=s\,\Ac' \dim M= t\dim \Ac'$ and using \eqref{dimension1} we obtain that $s\, \dim N=1$ whence $\dim N=1$ and the Theorem follows.
\epf

\section{A correspondence for twist equivalent Hopf
algebras}\label{twist-corresp} We shall present an explicit
correspondence between module categories over twist equivalent
Hopf algebras. For this we shall use the notion of biGalois
extension. A $(L,H)$-biGalois extension $B$, for two Hopf algebras
$L$, $H$ is a right $H$-Galois structure and a left $L$-Galois
structure on $B$ such that the coactions make $B$ an $(L,
H)$-bicomodule. For more details on this subject we refer to
\cite{S}.

\medbreak

Let $L$, $H$ be finite-dimensional Hopf algebras and $B$ a
$(L,H)$-biGalois extension. We denote by $\widetilde{B}$ the
$(H,L)$-biGalois extension with underlying algebra $B^{\op}$, and
comodule structure given as in \cite[Theorem 4.3]{S}. This new
biGalois extension satisfies that $B\square_H \widetilde{B}\simeq
L$ as $(L,L)$-biGalois extensions and $\widetilde{B}\square_H B
\simeq H$ as $(H,H)$-biGalois extensions. Here $\square_H$ denotes the cotensor product over $H$.

\medbreak

Let us recall that a Hopf 2-cocycle for $H$ is a  map $\sigma:
H\otk H\to \ku$, invertible with respect to convolution,  such
that for all $x,y, z\in H$
\begin{align}\label{2-cocycle}
\sigma(x\_1, y\_1)\sigma(x\_2y\_2, z) &= \sigma(y\_1,
z\_1)\sigma(x, y\_2z\_2),
\\
\label{2-cocycle-unitario} \sigma(x, 1) &= \varepsilon(x) =
\sigma(1, x).
\end{align}
Using this cocycle there is a new Hopf algebra structure
constructed over the same coalgebra $H$ with the product described
by
$$
x._{[\sigma]}y = \sigma(x_{(1)}, y_{(1)}) \sigma^{-1}(x_{(3)},
y_{(3)})\, \, x_{(2)}y_{(2)}, \qquad x,y\in H.
$$
This new Hopf algebra is denoted by $H^{\sigma}$. If $K$ is a left
$H$-comodule algebra, then we can define a new product in $K$ by
\begin{align}\label{sigma-product} a._{\sigma}b = \sigma(a_{(-1)},
b_{(-1)})\, a_{(0)}.b_{(0)},
\end{align}
$a,b\in K$. We shall denote by $K_{\sigma}$ this new left comodule
algebra. We shall say that the cocycle $\sigma$ is
\emph{compatible} with $K$ if for any $a, b\in K$, $\sigma(a\_2,
b\_2)\, a\_1 b\_1\in K$. In that case we shall denote by
${}_{\sigma}K$ the left comodule algebra with underlying space $K$
and algebra structure given by
\begin{align}\label{sigma-product2} a{}_{\sigma}.b = \sigma(a_{(-1)},
b_{(-1)})\, a_{(0)}.b_{(0)}\;\;\; a,b\in K.
\end{align}
The algebra $ H_\sigma$ is a left $H$-comodule algebra with
coaction given by the coproduct of $H$ and it is a $(H^{\sigma},
H)$-biGalois extension and ${}_\sigma H$ is a $(H,
H^{\sigma})$-biGalois extension.

\medbreak

If $\lambda:\Ac \to H\otk \Ac$ is a left $H$-comodule algebra then
$B\square_H \Ac$ is a left $L$-comodule algebra. The left coaction
is the induced by the left coaction on $B$ with the following
algebra structure, if $\sum x\ot a, \sum y\ot b \in B\square_H
\Ac$ then
$$(\sum x\ot a)(\sum y\ot b):=\sum xy\ot ab.$$
A direct computation shows that $B\square_H \Ac$ is a left
$L$-comodule algebra.

\begin{prop}\label{bigalois-c} The following assertions hold.
\begin{itemize}
    \item[1.] If $\Ac$ is right $H$-simple, then $B\square_H \Ac$
    is right $L$-simple.
    \item[2.] If $\Ac\sim_M \Ac'$ then $B\square_H \Ac\sim_M B\square_H \Ac'.$
    \item[3.] If $\sigma:H\otk H\to \ku$ is an invertible 2-cocycle and
    $L=H^{\sigma}$, $B=H_\sigma$ then $ B\square_H \Ac\simeq  \Ac_\sigma$.
\item[4.] If $K\subseteq H$ is a left coideal subalgebra,
$\tau:H\otk H\to \ku$ is an invertible 2-cocycle compatible with
$K$, $\sigma:H\otk H\to \ku$ is an invertible 2-cocycle and
$B=H_\sigma$ then $B\square_H \big({}_\tau\! K\big)\simeq \big(
{}_\tau K\big)_\sigma$.
\end{itemize}
As a consequence we obtain that the application $\Ac\to B\square_H
\Ac$ gives a explicit bijective correspondence between
indecomposable exact module categories over $\Rep(H)$ and over
$\Rep(L)$.
\end{prop}

\pf 1. If $I\subseteq \Ac$ is a right ideal $H$-costable then
$B\square_H I$ is a right ideal $L$-costable of $B\square_H \Ac$.

2. Let $P\in {}^{H\!}\Mo_{\Ac}$ such that $\Ac'\simeq
\End_{\Ac}(P)$ as comodule algebras. The object $B\square_H P$
belongs to the category ${}^{L\!}\Mo_{B\square_H\Ac}$. The result
follows since there is a natural isomorphism
$B\square_H\End_{\Ac}(P)\simeq \End_{B\square_H \Ac}(B\square_H
P)$.

3. and 4. follow by a straightforward computation.\epf

\subsection{BiGalois extensions for quantum linear spaces}

Let $\theta\in \Na$ and $\Gamma$ be a finite Abelian group,
$(g_1,\dots, g_{\theta},\chi_1, \dots,\chi_{\theta})$ be a datum
of a quantum linear space, $V=V(g_1,\dots, g_{\theta},\chi_1,
\dots,\chi_{\theta})$ and $\qs$ the quantum linear space
associated to this data. Let $U= \qs\#\ku\Gamma$.

Let $\Do=(\mu, \lambda)$ be a compatible datum for $\qs$ and
$\Gamma$, and $H=\Ac(\Gamma,\qs,\Do)$ be the Hopf algebra as
described in section \ref{liftings}. We shall present a
$(H,U)$-biGalois object.

The pair $(-\mu,-\lambda)$ is a compatible comodule algebra datum
with respect to $\qs$ and the trivial 2-cocycle. In this case the
left $U$-comodule algebra $\Ac(V,\Gamma,1,-\mu,-\lambda)$ is also
a right $H$-comodule algebra with structure
$\rho:\Ac(V,\Gamma,1,-\mu,-\lambda)\to
\Ac(V,\Gamma,1,-\mu,-\lambda)\otk H$ determined by
$$ \rho(e_g)=e_g\ot g,\quad \rho(v_i)=v_i\ot 1+ e_{g_i}\ot a_i,
\;\; g\in \Gamma,\;\;  i=1,\dots,\theta.$$

The following result seems to be part of the folklore.

\begin{prop}\label{cocy-def} The algebra $\Ac(V,\Gamma,1,-\mu,-\lambda)$ with the above
coactions is a $(H,U)$-biGalois object.
\end{prop}
\pf Straightforward. \epf

\end{document}